\documentclass[a4paper,11pt]{article}
\usepackage{amsmath}
\usepackage{amsfonts}
\usepackage{amssymb}
\usepackage{amsthm}
\usepackage{textcomp}
\usepackage{tipa}
\usepackage{fullpage}
\usepackage{graphicx}
\usepackage{color}
\usepackage{url}
\newcommand{\comments}[1]{}

\renewcommand{\leq}{\leqslant}

\renewcommand{\geq}{\geqslant}

\newcommand{\w}{\omega}

\newcommand{\ka}{\kappa}

\newcommand{\R}{\mathbb{R}}

\newcommand{\N}{\mathbb{N}}

\newcommand{\B}{\mathcal{B}}

\newcommand{\Lind}{Lindel\"of}

\providecommand{\hdf}[1]{$#1$-Hausdorff}
\providecommand{\n}[1]{\frac{1}{#1}}

\swapnumbers
\newtheorem{thm} {Theorem}[section]

\theoremstyle{definition}
\newtheorem{defn}[thm]{Definition}
\newtheorem{xmpl}[thm]{Example}

\newtheorem*{oq}{Open Question}
\newtheorem*{ackn}{Acknowledgements}

\theoremstyle{remark}

\newtheorem{notation}[thm]{Notation}

\title{A Note on the Hausdorff Number of Compact Topological Spaces }
\author{Petra Staynova}
\date{}

\begin{document}

\maketitle

\begin{abstract}
The notion of Hausdorff number of a topological space is first introduced in \cite{bonan}, with 
the main objective of using this notion to obtain generalizations of some known bounds for cardinality of topological spaces. 
Here we consider this notion from a topological point of view and examine interrelations of the Hausdorff number with compactness. 
\\
\\
2000 Math.Subj.Classification: 54D10, 54D30
\end{abstract}

\section{Introduction}

\begin{defn} \cite{bonan} 
Let $X$ be a topological space. 
Let $
H(X)=\min\{\tau:\forall A\subset X: |A|\geq \tau, \forall a\in A, \textrm{there exist open neighborhoods } 
U_a\ni a \textrm{ such that } \bigcap_{a\in A} U_a=\emptyset\}.
$
This number (finite or infinite) is called the \emph{Hausdorff number} of $X$. 
\end{defn}
It is clear that if $X$ contains at least two points then $X$ is Hausdorff if and only if $H(X)=2$. 

\begin{defn}
If $n\in\N$, $n\geq2$, then $X$ is called \emph{$n$-Hausdorff} if $H(X)\leq n$. 
\end{defn}
In \cite{bonan} it was pointed out that the property of being $n$-Hausdorff is independent from the $T_1$ property. 
Also, it is clear that if $X$ is $n$-Hausdorff then it is $(n+1)$-Hausdorff for any $n\in\N, \ n\geq2$, hence all Hausdorff spaces are trivially $n$-Hausdorff for any $n\geq 2$. 
For every finite $n$, examples of $(n+1)$-Hausdorff spaces that are not $n$-Hausdorff are also considered in \cite{bonan}. 

Since the Hausdorff property nicely correlates with some basic covering properties, such as compactness and \Lind ness \cite{Eng}, here we consider a natural question arising in this setting. 
It is well-known that every Hausdorff compact space is regular and normal, i.e. the combination of compact with the Hausdorff separation axiom leads to stronger separation properties. 
Also it is well known that any Hausdorff regular Lindelof space is normal \cite{Eng}. 

For any $n\geq 2$, we shall give examples of $(n+1)$-Hausdorff spaces that are not $n$-Hausdorff and in addition are $T_1$, compact and first countable. 
Also we shall give an example of a compact $T_1$ first countable $\w$-Hausdorff space which is not $n$-Hausdorff for any $n\geq2$. 
Our examples will be neither regular nor normal, thus showing that even in ``nice'' topological spaces, 
 \hdf{n}ness is something much weaker than the Hausdorff property. 
Finally, some open questions in the above vein will be posed. 
The examples that we present are different from those given in \cite{bonan} and, in our point of view, much more transparent. 
Also, all the examples in \cite{bonan} are countable, while most of the examples we give here are uncountable. 

\subsection{Notation and Assumptions}

Throughout this exposition, a space $X$ will be considered compact if from any open cover of $X$ one can choose a finite subcover. 

\begin{notation}
We use:
\begin{itemize}
\item $\langle a, b \rangle$ to denote the ordered pair,
\item $(a,b)$ to denote the open interval on the real line, 
\item $k, l, m, n$ to denote positive integers.
\end{itemize}
\end{notation}
\section{The Finite Case}

Most mathematicians might believe that study of topologies on finite sets is a trivial pursuit, but this is far from true. 
Questions concerning topologies on spaces containing finitely many points can be very difficult, and have been considered as early as 1966, \cite{Stong}. 
There are still many open questions, such as how many (non-homeomorphic) topologies there are on a set of $n$ elements, and how many $T_1$ topologies there are on an $n$-point set. 
So far, the best partial answers to these questions give approximations, for example  \cite{benoumhani}, \cite{Tenner}, \cite{Dorsett} and \cite{Berrone}. 


We consider the following examples and statements in view of the fact that all Hausdorff topologies on a finite set are discrete, \cite{workbook}. 

\begin{xmpl}
A simple example of a compact 3-Hausdorff space which is not Hausdorff.
\end{xmpl}
On the underlying set $X=\{x,y,z\}$, we introduce the topology
$$
\tau=\{\emptyset, \{x\},\{y,z\},\{x,y,z\}\}.
$$
It is easy to check that this is indeed a topology. 
It is also \hdf{3}, because the only subset of cardinality $3$ is $\{x,y,z\}$, and $\{x\},\ \{y,z\}$ form a disjoint open cover. 
It is also compact, since it is finite. 

Hence, in compact, and even finite, spaces the property of being \hdf{3} does not coincide with Hausdorfness.

This leads us to the following interesting theorem:

\begin{thm}
For each $n\in\N$, 
there exists an \hdf{n} topology on the $n$-point set which is not discrete.
\end{thm}

\begin{proof}
Let $X$ be a set with $n$ points. 
Pick $x_0\in X$, and take the topology $\{\emptyset, \{x_0\}, X\setminus\{x_0\}, X \}$. 
\end{proof}

\begin{xmpl}\label{4point}
There is a \hdf{3} space of cardinality $4$, which is not discrete. 
\end{xmpl}
Let $X=\{w,x,y,z\}$, and let $F$ be the filter generated by $\{y\}$. 
Then the topology 
$$
\tau=F\cup\{\emptyset, \{x\},\{z\}, \{x,z\}\}
$$
is \hdf{3} and not discrete. 

On the basis of this example we can establish the following more general result:
\begin{thm}\label{3hdf}
For any $n\geq 3$, there is a \hdf{3} topology on an $n$-point set $X$, which is not discrete. 
\end{thm}
\begin{proof}
Let $X$ be a space with $n$ points, and choose two `special' points, $x_0, x_1 \in X$. 
Then the basis
$$
\B=\{\{x\}:x\in X\setminus \{x_0\}\}\cup \{x_0,x_1\}
$$
is a basis for a non-discrete \hdf{3} topology on $X$. 
Indeed, any 3-element subset of $X$ can be separated via basic open sets. 
Let $A\subset X, |A|=3$. 
Then either $x_0\in A$ or $x_0 \notin A$. 
If $x_0\in A$, the family $\big\{\{x\}:x\in A\setminus\{x_0\}\big\}\cup \big\{\{x_0,x_1\}\big\}$ is the required family of basic open sets with empty intersection. 
If $x_0 \notin A$, the family $\big\{\{x\}:x\in A\big\}$ will separate the points of $A$. 
\end{proof}

Some recent results on the number of topologies on a finite set with $n$ elements have been obtained in \cite{Berrone}, \cite{Tenner}, \cite{benoumhani}, and \cite{Dorsett}. 
In the archive (\url{http://arxiv.org}), one can find even more recent papers (from 2012) that also treat these questions. 
Theorem \ref{3hdf} justifies asking similar questions for \hdf{n} topologies. 
The following seem interesting questions, though some might end up rather easy:
\begin{oq}
In \cite{markowsky}, several formulae are given which relate the number of $T_0$ topologies on an $n$-point set with the number of topologies on $n$-points set. 
Can similar formulae be found for the number of \hdf{k} topologies on an $n$-point set, where $k\leq n$?
\end{oq}
Or, in general:
\begin{oq}
How many distinct \hdf{k} topologies are there on a finite set $X$ of cardinality $n \in \N$, for $k\leq n$?
\end{oq}
Moreover:
\begin{oq}
For $k,l<n$, is there a formula relating the number of \hdf{k} topologies on a space of cardinality $n$ to the number of \hdf{l} topologies on the same space? 
\end{oq}
\comments{
\begin{oq}
If a space $X$ has finite cardinality $n$, can we find a `limit number' $k<n$ such that, for all $l\leq k$, the \hdf{l} topology is the discrete? We have a natural `upper bound' in $n$ and natural `lower bound' in $2$ (since every Hausdorff topology on a finite space is discrete), but Example \ref{4point} shows that we can have a non-discrete \hdf{k} topology for $2<k\lneqq n$. 
\end{oq}
}

\section{Main Examples}
Let us now consider the uncountable cases. 
The following examples were inspired by the question, discussed in the introduction, of whether any compact \hdf{3} space is Hausdorff (and therefore regular and normal). 

\begin{xmpl}\label{xmpl1}
There is a $T_1$, uncountable, \hdf{3} space that is compact first countable but not Hausdorff, regular, or normal. 
\end{xmpl}
Define $X=([0,1]\times\{0\})\cup\{\langle\n{2},\n{2}\rangle\}$. 
Topologize $X$ as follows:
\begin{itemize}
\item all points on $[0,1]\times\{0\}$ have the Euclidean neighborhoods;
\item the neighborhoods of $\{\langle\n{2},\n{2}\rangle\}$ consist of 
$$
U_n\left(\left\langle\n{2},\n{2}\right\rangle\right)=
\left\{\left\langle\n{2},\n{2}\right\rangle\right\}\cup
\left(\left(\left(\n{2}-\n{n},\n{2}+\n{n}\right)\setminus\left\{\n{2}\right\}\right)
\times\left\{0\right\}\right). 
$$
\end{itemize}
It is clear that $X$ is compact first countable (since this is a compact space with one additional point added and all points have  countable neighbourhood basis). 
It is also clear that $X$ is $T_1$, since $[0,1]$ is $T_1$ (even Hausdorff, regular and normal), and 
$$
\bigcap_{n\in\N} U_n\left(\left\langle\n{2},\n{2}\right\rangle\right)=
\left\{\left\langle\n{2},\n{2}\right\rangle\right\}.
$$
It is also clear that this space is not Hausdorff, since $\left\langle\n{2},\n{2}\right\rangle$ and $\left\langle\n{2},0\right\rangle$ do not have disjoint neighborhoods. 
It is \hdf{3} because if we have any three different points in $X$, 
at least two would be in $[0,1]\times\{0\}$, and since the latter is Hausdorff in the Euclidean topology, those two points would have disjoint neighborhoods. 

Since $X$ is $T_1$, we again have that $\left\langle\n{2},\n{2}\right\rangle$ and $\left\langle\n{2},0\right\rangle$, being closed, witness the fact that $X$ is neither regular nor normal.

If we modify Example \ref{xmpl1} by defining 
$$
U_n\left(\left\langle\n{2},\n{2}\right\rangle\right)=
\left\{\left\langle\n{2},\n{2}\right\rangle\right\}
\cup
\left(
\left(\n{2}-\n{n},\n{2}+\n{n}\right)\times\{0\}
\right),
$$
then this new space has the same properties as before, except it fails to be $T_1$. 
Hence even a \hdf{3} space does not entail $T_1$, and we have:
\begin{xmpl}\label{xmpl2}
There exists a \hdf{3} compact uncountable first countable space that is not $T_1$. 
\end{xmpl}
The above idea can easily be generalized for any finite $n\geq 2$:
\begin{xmpl}\label{xmpl3}
There is a $T_1$ uncountable compact first countable \hdf{(n+1)} space which is not \hdf{n} (hence not regular or normal). 
\end{xmpl}
As an underlying set, we take $X=([0,1]\times\{0\})\cup\left\{\left\langle\n{2},\n{m}\right\rangle:m<n\right\}$, and 
topologize $X$ as follows:
\begin{itemize}
\item all points on $[0,1]\times\{0\}$ again have the Euclidean neighborhoods;
\item the neighborhoods of $\{\langle\n{2},\n{m}\rangle\}$ consist of 
$$
U_k\left(\left\langle\n{2},\n{m}\right\rangle\right)=
\left\{\left\langle\n{2},\n{m}\right\rangle\right\}\cup
\left(\left(\left(\n{2}-\n{k},\n{2}+\n{k}\right)\setminus \left\{\n{2}\right\}\right)
\times\left\{0\right\}\right). 
$$
\end{itemize}

\begin{xmpl}\label{xmpl4}
There is an \hdf{\w_1} uncountable compact first countable $T_1$ space $X$ which is not $n$-Hausdorff for any finite $n\geq2$, and is not \hdf{\w}. 
\end{xmpl}


We take $X=([0,1]\times\{0\})\cup\left\{\left\langle\n{2},\n{m}\right\rangle:m\in\w\right\}$, and 
topologize $X$ as follows:
\begin{itemize}
\item all points on $[0,1]\times\{0\}$ have neighborhoods which are traces in $X$ of their respective Euclidean neighborhoods in $\R^2$;
\item the neighborhoods of $\{\langle\n{2},\n{m}\rangle\}$ consist of 
$$
U_k\left(\left\langle\n{2},\n{m}\right\rangle\right)=
\left\{\left\langle\n{2},\n{m}\right\rangle\right\}\cup
\left(\left(\left(\n{2}-\n{k},\n{2}+\n{k}\right)\setminus \left\{\n{2}\right\}\right)
\times\left\{0\right\}\right). 
$$
\end{itemize}
This space is compact since it consists of the union of two compact sets:  
of a  sequence converging in our topology, and of the closed unit interval with the Euclidean topology, 
which is is also compact in our topology. 
In fact, this space has a separation property which is much stronger than \hdf{\w_1}, 
since the only countable subset of it which does not have neighborhoods with empty intersections is the vertical sequence 
$\left\{\left\langle\n{2},\n{m}\right\rangle:m\in\w\right\}$. 
Otherwise, if the countable set contains at least one point in $([0,1]\setminus\{\n{2}\})\times \{0\}$, 
then it obviously has neighborhoods with empty intersection (since $[0,1]\times\{0\}$ is Hausdorff). 
\section{Inheritance of Hausdorff number by subspaces}
It is well-known that the Hausdorff property is inherited by any subspace of a topological space. 
It is clear that the same is true for the Hausdorff number. 
But some subspaces of \hdf{n} spaces can have a Hausdorff number less than $n$. 
If a subspace is a two-point set, it will of course be Hausdorff. 
So, it is interesting to make the following observations about some uncountable subspaces of our examples:
\begin{itemize}
\item Example \ref{xmpl1} has an uncountable open compact subspace which is Hausdorff, namely $[0,1]\times\{0\}$;
\item for each $n\in\w$, Example \ref{xmpl4} has an uncountable open compact subspace which is \hdf{(n+1)} but not \hdf{n}: 
for a fixed $n\in\w$, take the subspace $Y_n=([0,1]\times\{0\})\cup\left\{\left\langle\n{2},\n{m}\right\rangle:m<n\right\}$. 
Then $Y_n$ is not \hdf{n}, but it is \hdf{(n+1)} (and is also compact and open). 
\end{itemize}


Hence the Hausdorff number of a subspace is less than or equal to that of the space. 
This leads us to the following natural question:
\begin{oq}
Given an uncountable \hdf{\tau} space $X$, do we have that for each $\ka<\tau$, $X$ has an uncountable \hdf{\ka} subspace? 
\end{oq}

\begin{ackn}
The author is grateful to the referee for carefully reading the paper, and for giving helpful suggestions for examples \ref{xmpl3}, \ref{xmpl4}, and for improving the general exposition. 
\end{ackn}
\noindent
contact email: petra.staynova@gmail.com  \\
Mailing Address:\\
Department of Mathematics\\
University of Leicester\\
LE1 7RH\\
Leicester\\
United Kingdom

\bibliographystyle{plain}
\bibliography{Bibliography}

\end{document}